\documentclass[12pt]{amsart}

\usepackage{graphicx}
\usepackage{amsfonts}
\usepackage{epsf}
\usepackage{amssymb}
\usepackage{amsmath}
\usepackage{amscd}
\usepackage{hyperref}
\usepackage{color}
\usepackage{bbm}

\newtheorem{theorem}{Theorem}%[section]

\newtheorem{corollary}[theorem]{Corollary}

\theoremstyle{remark}

\newtheorem{example}[theorem]{Example}

\input xy
\xyoption{all}

\begin{document}

\title{Periodic spanning surfaces of periodic knots}
\begin{abstract}
Edmonds \cite{Edmonds} famoulsy proved that every periodic knot of genus $g$ possesses an equivariant Seifert surface of genus $g$. We show that this is not true if one instead considers nonorientable spanning surfaces of a periodic knot. We demonstrate by example that the difference between the first Betti number of an equivariant and a nonequivariant nonorientable spanning surface of a periodic knot, can be arbitrarily large.   
\end{abstract}
%
%\subjclass[2010]{57M25 and 57M27} 
%
\thanks{The author was partially supported by the Simons Foundation, Award ID 524394, and by the NSF, Grant No. DMS--1906413. }
\author{Stanislav Jabuka}
\email{jabuka@unr.edu}
\address{Department of Mathematics and Statistics, University of Nevada, Reno NV 89557, USA.}
\maketitle
%%%%%
%%%%%
%%%%%
%%%%%
%%%%%
%%%%%
%\section{Introduction}
%%%
%%%
A knot $K$ in $S^3$ is said to be {\em periodic} if there exists an integer $p\ge 2$, a diffeomorphism $f:S^3\to S^3$ of order $p$ that preserves the knot $K$, and whose fixed point set Fix$(f)$ is diffeomorphic to $S^1$. In this case we say that $K$ is $p$-periodic, that $p$ is a period of $K$,  and we call Fix$(f)$ the {\em axis of $f$}. %The positive resolution of the Smith Conjecture guarantees that Fix$(f)$ is an unknot in $S^3$.  
See \cite{JabukaNaik} for more background on periodic knots. 

In \cite{Edmonds} Edmods famously proved that if $K$ is a $p$-periodic knot of genus $g$, then there exists a Seifert surface $\Sigma$ for $K$ of genus $g$ that is invariant under the diffeomorphism $f$. Said differently, if we define the {\em $p$-periodic} (or {\em equivariant}) {\em 3-genus $g_{3,p}(K)$ of a $p$-periodic knot $K$} as 
$$g_{3,p}(K) = \min \{ g\ge 0 \, |\,   \text{ $K$ possesses an $f$-invariant Seifert surface of genus $g$}  \},$$ 
then Edmonds' theorem can be seen as saying that $g_3(K) = g_{3,p}(K)$ for every $p$-periodic knot $K$ (with $g_3(K)$ being the Seifert genus of $K$). %A corollary of this result is the bound $p\le 2g_3(K)+1$ satisfied by any period $p$ of the knot $K$. While it was known prior to Edmonds' work that a knot may only have finitely many periods (cf. Theorem 3 in \cite{Flapan}),  the preceding inequality was the first quantitative bound on the number of possible periods of a knot.  

The goal of this note is to show that if one considers nonorientable spanning surfaces for periodic knots instead, the analogue of Edmonds' theorem is not true. To state our result, we recall the definition of the {\em nonorientable (nonequivariant) 3-genus $\gamma_3(K)$}, and we define the {\em $p$-periodic} (or {\em equivariant}) {\em  nonorientable 3-genus $\gamma_{3,p}(K)$} of a $p$-periodic knot $K$ as 
\begin{align*}
 \gamma_{3}(K) & = \min \{ b_1(\Sigma)  \, |\,   \text{ $\Sigma$ is a nonorienatble spanning surface for  $K$}  \}, \cr
 \gamma_{3,p}(K) & = \min \{ b_1(\Sigma)  \, |\,   \text{ $\Sigma$ is an $f$-invariant nonorienatble spanning surface for $K$}  \}.
\end{align*}
We leave it as an easy exercise to show that every periodic knot has an equivariant nonorientable spanning surface, and thus the definition of $\gamma_{3,p}(K)$ is well posed. It is also not hard to show that $\gamma_{3,p}(K) \le 2g_3(K) +p$ for any $p$-periodic knot $K$.  
%%%
%%% 
\begin{theorem} \label{main}
Let $K$ be a $p$-periodic knot with $p\ge 3$ and with $\gamma_3(K)\ge 2$. Then $\gamma_{3,p}(K) \ge p$. 
\end{theorem}
%%%
%%%
\begin{proof}
Let $f:S^3\to S^3$ be a diffeomorphism that facilitates the $p$-periodicity of $K$ and let $A=$Fix$(f)$ be its axis. Let further $\Sigma \subset S^3$ be a nonorientable $f$-invariant spanning surface for $K$ and let $\overline \Sigma \subset S^3$ be the quotient of $\Sigma$ by the action of $\mathbb Z_p$ generated by $f$, and note that $\overline \Sigma$ is nonorientable. Then $\Sigma\to \overline \Sigma$ is a $p$-fold cyclic cover, branched along $\lambda \ge 0$ points, with $\lambda$ being the number of points in $\Sigma \cap A$. A straightforward computation of Euler characteristics  gives 
\begin{equation}  \label{EquationEulerCharacteristics}
\chi(\Sigma)   = p\cdot \chi (\overline \Sigma)  - (p-1)\lambda.
\end{equation}   
Write $b_1(\Sigma) = a$ and $b_1(\overline \Sigma) = b$. The assumption $\gamma_3(K) \ge 2$ forces $a\ge 2$, while by definition $b\ge 1$ and $\lambda \ge 0$. Equation \eqref{EquationEulerCharacteristics} then becomes 
\begin{equation} \label{EquationOfBAndLambda}
a-1 = p(b-1) +(p-1)\lambda.
\end{equation}
If $b=1$, we obtain $a-1 = (p-1)\lambda$ forcing $\lambda >0$ since $a\ge 2$. This in turn forces the inequality $a-1\ge p-1$ or $a\ge p$. If $b\ge 2$ then \eqref{EquationOfBAndLambda} implies $a-1\ge p$. Thus, in either case we find $a\ge p$ and hence $\gamma_{3,p}(K_p) \ge p$, since $\Sigma$ was an arbitrary equivariant nonorientable spanning surface for $K$. 
\end{proof} 
%%%
%%%
\begin{corollary} \label{CorollaryOne}
The difference between the equivariant and nonequivariant nonorientable 3-genera of a periodic knot can become arbitrarily large. Specifically, for every integer $p\ge 3$ there exists a $p$-periodic knot $K_p$ with 
$$\gamma_3(K_p) = 2 \qquad \text{ and } \qquad \gamma_{3,p}(K_p) \ge p.$$
\end{corollary}
%%%
%%%
\begin{proof}
Let $K_p$ be the torus knot $T(4p,2p-1)$. By \cite{Teragaito} (see also \cite{JabukaVanCott}) we obtain $\gamma_3(K_p)=2$ for all $p\ge 3$. The periods of a torus knot $T(a,b)$ are precisely the divisors of $|a|$ and $|b|$, showing that $K_p$ is $p$-periodic. Theorem \ref{main} implies that $\gamma_{3,p}(K_p) \ge p$. 
\end{proof} 
%%%
%%%
The inequality $\gamma_{3,p}(K) \ge p$ from Theorem \ref{main} is sharp as seen in the next example. 
%%%
%%%
\begin{example}
Consider the 5-periodic torus knot $K=T(5,3)$. It follows from \cite{Teragaito} that $\gamma_3(K) = 2$ (or use \cite{KnotInfo} where $T(5,3)$ is the knot $10_{124}$), showing that $K$ meets the hypothesis of Theorem \ref{main} and thus $\gamma_{3,5}(K) \ge 5$. An equivariant spanning surface $\Sigma$ for $K$ with $b_1(\Sigma) = 5$ is shown in Figure \ref{FigureFor12-124}, leading to $\gamma_{3,5}(K) = 5$. The values of $a$, $b$, $\lambda$ from the proof of Theorem \ref{main} are 5, 1, 1 respectively, and satisfy equation \eqref{EquationOfBAndLambda}. 
\end{example}
%%%
%%% 
\begin{figure}
\includegraphics[width=5cm]{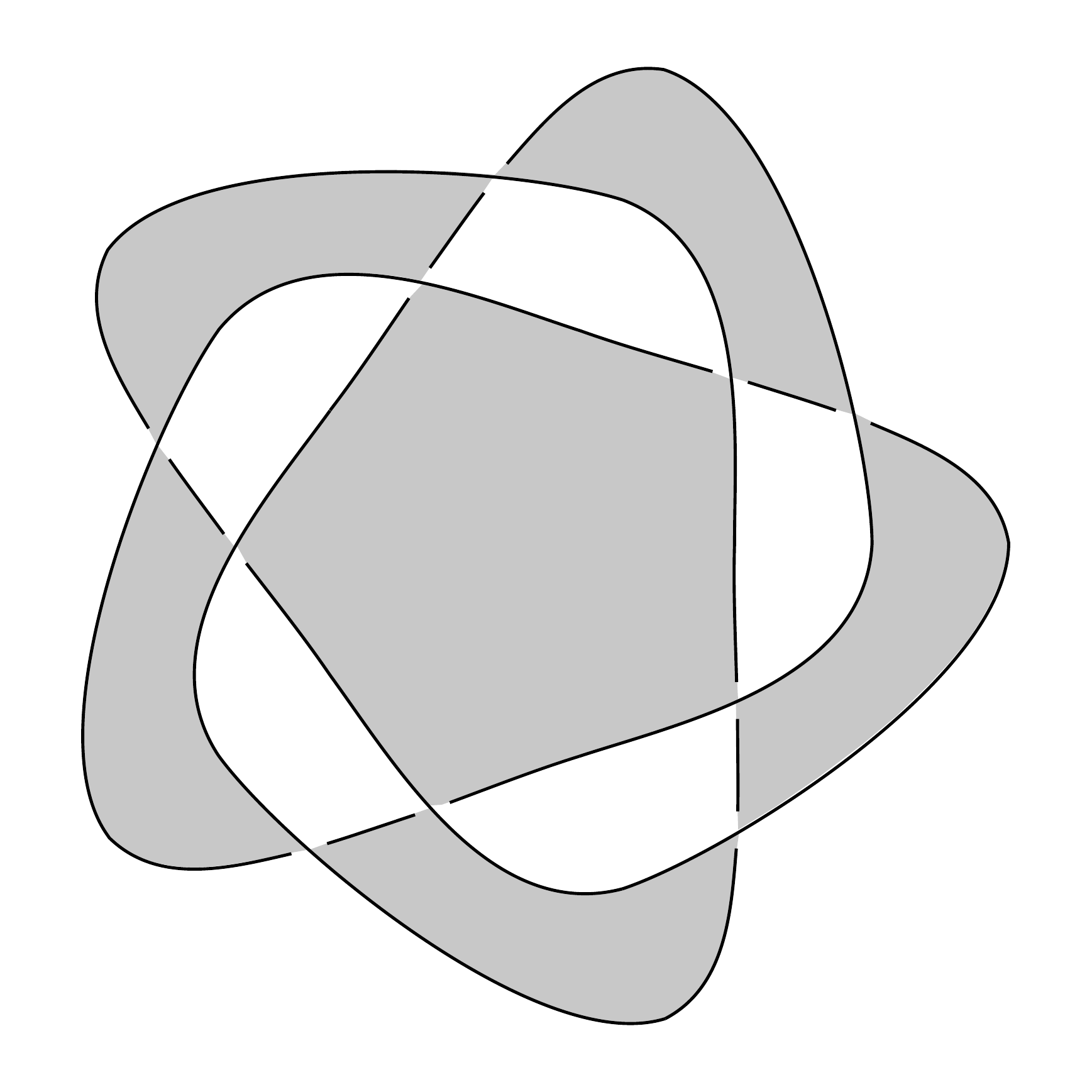}
\caption{The torus knot $T(5,3)$ shown with an equivariant nonorientable spanning surface $\Sigma$ with $b_1(\Sigma ) = 5$.   } \label{FigureFor12-124}
\end{figure}
%%%
%%%
Another important result of Edmonds' \cite{Edmonds} is the bound $p\le 2g_3(K)+1$ satisfied by any period $p$ of the knot $K$. While it was known prior to Edmonds' work that a knot may only have finitely many periods (cf. Theorem 3 in \cite{Flapan}),  the preceding inequality was the first quantitative bound on the number of possible periods of a knot. Corollary \ref{CorollaryOne} shows, as yet another contrast to Edmonds' results, that no upper bound on the periods of a knot can exist by any polynomial function in the nonorientable 3-genus. This conclusion also follows from considering the $p$-periodic alternating torus knots $T(2,p)$ for which $\gamma_3(T(2,p)) = 1 = \gamma_{3,p}(T(2,p))$, with $p\ge 3$ odd. 
%%%
%%%
%\begin{corollary} \label{CorollaryTwo}
%There is no upper bound on the periods of periodic knots by a polynomial function in the nonorientable 3-genus.  
%\end{corollary}
%%%
%%%
%Corollary \ref{CorollaryTwo} follows from Corollary \ref{CorollaryOne}, but can also be deduced by considering the $p$-periodic alternating torus knots $T(2,p)$. It is easy to show that  $\gamma_3(T(2,p)) = 1 = \gamma_{3,p}(T(2,p))$.  
%%%%%%%%%%%%%%%
%%%%%%%%%%%%%%%
\bibliographystyle{plain}
\bibliography{bibliography}

\end{document}